\documentclass[a4paper,10pt]{amsart}
\usepackage[plainpages=false]{hyperref}
\usepackage{amsfonts,amssymb,amsthm, mathrsfs, lscape}
\usepackage{verbatim}

\usepackage{latexsym}
\usepackage{amssymb}
\usepackage{graphics}
\usepackage{graphicx}
\usepackage[space]{cite}

\usepackage{latexsym}
\usepackage{amsfonts}
\usepackage{amsmath}
\usepackage{latexsym}
\usepackage{amsfonts}
\usepackage{amssymb}
\usepackage{rawfonts}

\usepackage[usenames,dvipsnames]{pstricks}
\usepackage{epsfig}
\usepackage{pst-grad} 
\usepackage{pst-plot} 
\usepackage[space]{grffile} 
\usepackage{etoolbox} 
\usepackage{bm}
\makeatletter 

\renewcommand{\Re}{{\operatorname{Re}\,}}
\renewcommand{\Im}{{\operatorname{Im}\,}}

\renewcommand{\epsilon}{\varepsilon}

\newcommand{\kahler}{K\"{a}hler }

\newcommand{\N}{{\mathbb N}}

\newcommand{\R}{{\mathbb R}}
\newcommand{\C}{{\mathbb C}}

\newcommand{\Z}{{\mathbb Z}}

\newcommand{\vol}{{\operatorname{Vol}}}
\newcommand{\iddbar}{\sqrt{-1}\partial \bar{\partial}}
\newcommand{\Pic}{{\operatorname{Pic}}}

\newcommand{\eqd}{\buildrel {\operatorname{def}}\over =}

\newcommand{\hcal}{\mathcal{H}}

\newcommand{\jcal}{\mathcal{J}}

\newcommand{\ocal}{\mathcal{O}}

\newcommand{\ucal}{\mathcal{U}}
\newcommand{\vcal}{\mathcal{V}}

\newtheorem{thm}{Theorem}[section]
\newtheorem{theorem}{{Theorem}}[section]

\newtheorem{lemma}[theorem]{Lemma}
\newtheorem{proposition}[theorem]{Proposition}
\newtheorem{conjecture}[theorem]{Conjecture}

\newenvironment{remark}{\medskip\noindent{\it Remark:\/} }{\medskip}

\newtheorem{question}[thm]{Question}

\newtheorem{defandthm}[theorem]{Definition and Theorem}

\theoremstyle{empty}

\numberwithin{equation}{section}
\newcommand{\val}{{\operatorname{Val}}}
\def \N {\mathbb N}

\def \C {\mathbb C}
\def \Z {\mathbb Z}
\def \R {\mathbb R}

\def \O {\mathcal O}

\def \Hom {\text{Hom}}
\title[Bergman Kernels of abelian varieties]{On the Bergman kernel of polarized abelian varieties}

\author{Jingzhou Sun}
\address{Department of Mathematics, Shantou University, Shantou City, Guangdong Province 515063, China}
\email{jzsun@stu.edu.cn}

\begin{document}

	\begin{abstract}
We prove an explicit formula for the Bergman kernel of polarized abelian varieties. As applications, we show that if two positive line bundles represent the same first Chern class and have identical Bergman kernel functions for some tensor power, then the corresponding powers of the bundles are isomorphic. We also obtain localization results for the maxima and minima of the Bergman kernel function and establish exponential off-diagonal decay estimates.
	\end{abstract}
	
	\maketitle
	\setcounter{tocdepth}{1}
	 \tableofcontents
\section{Introduction}
Let $(M,\omega)$ be a \kahler manifold of complex dimension $n$, and let $L\to M$ be a holomorphic line bundle equipped with a Hermitian metric $h$ whose Chern curvature equals $-i\omega$. The Bergman space $\hcal_k$ consists of holomorphic sections of $L^k$ that are square integrable:
\[
\int_M |s|_{h}^2\,\frac{\omega^n}{n!}<\infty.
\]
With the $L^2$--inner product induced by $h$ and $\omega$, $\hcal_k$ is a closed subspace of $L^2(M,L^k)$. The Bergman kernel $K_k(x,y)$ is the integral kernel of the orthogonal projection $L^2(M,L^k)\to\hcal_k$, characterized by
\[
s(y)=\int_M \langle s(x),K_k(x,y)\rangle_{h^k}\,\frac{\omega^n(x)}{n!}\qquad(\forall s\in\hcal_k).
\]
The diagonal density (or Bergman kernel function) is the pointwise trace of $K_k$,
\[
\rho_k(x)=|K_k(x,x)|_{h^k}.
\]

For trivial bundles over domains in $\C^n$ one uses the analogous definition for holomorphic $L^2$--functions.

The Bergman kernel has been intensively studied for domains in $\C^n$ and for \kahler manifolds (see, e.g., \cite{Fefferman1974,boutet1975singularite,kerzman1978cauchy,donaldson2001,Sun2011Expected,Donaldson2014Gromov,berman2011fekete,Donaldson15, dsz1,dsz2,dsz3,Shiffman2008NV,sz1,bsz0,bsz1}). Except in very symmetric cases (for instance complex space forms), explicit formulas are rare. A fundamental substitute is the Tian-Catlin-Zelditch expansion (\cite{Tian1990On, Zelditch2000Szego, Lu2000On, Catlin}): when $M$ is compact,
\[
\rho_k(x)\sim\Big(\frac{k}{2\pi}\Big)^n\Big(1+\frac{b_1(x)}{k}+\frac{b_2(x)}{k^2}+\cdots\Big),
\]
where each $b_j(x)$ is a universal polynomial in the curvature of the Riemannian metric corresponding to $\omega$ and its covariant derivatives. There are extensions of this expansion to symplectic manifolds and orbifolds (see \cite{MM,dailiuma}).

In \cite{sun-rem-hyp}, the author made a progress by proving an explicit formula for the Bergman kernel function for positive line bundles on hyperbolic Riemann surfaces. The formula involves summation over all geodesic loops based at the point where the Bergman kernel is evaluated.
 In this paper, inspired by this work, we continue to study the Bergman kernel of abelian varieties.

\

Let $X=\C^n/\Lambda$ be an abelian variety over $\C$. Here, $\Lambda$ is a lattice in $\C^n$. Let $L\to X$ be a positive line bundle over $X$, whose first Chern class is determined by a positive definite Hermitian form $H=(H_{ij})$ on $\C^n$ such that $\Im H(\Lambda,\Lambda)\subset \Z$. For each vector $v\in \C^n$, we denote by $|v|_H$ the norm of $v$ defined by $H$.
Let $\omega=\pi \sqrt{-1} \sum H_{ij} dz^i\wedge d\bar{z}^j \in 2\pi c_1(L)$ be a flat K\"ahler metric on $X$. There exist a Hermitian metric $h$ on $L$ with curvature $- i \omega$.

For each $p\in X$, $\exists \tilde{p}\in\C^n$ such that $p=\pi(\tilde{p})$. Such a $\tilde{p}$ is called a lift of $p$.
For each $v\in \Lambda$, let $\gamma_{p,v}\subset X$ be the smooth geodesic loop defined by $v$ starting at $p$ with orientation defined by $\frac{\partial}{\partial a}$, where $a$ is the parameter of the line segment $\tilde{p}+av$, $a\in[0,1]$. It is easy to see that $\gamma_{p,v}$ does not depend on the choice of $\tilde{p}$, and that the length of $\gamma_{p,v}$ is $|v|_H$. Let $\nabla$ be the Chern connection of $(L,h)$.
 Let $e^{2\pi i\alpha_v(p)}=\text{Hol}_{L}(\gamma_{p,v})$ be the holonomy of $\nabla$ along $\gamma_{p,v}$. 
Our first result is the following formula:
\begin{theorem}
\label{thm-main}
	 For $k\geq 1$, the Bergman kernel $\rho_k$ of $H^0(X,L^k)$ at each $p\in X$ satisfies 
	 \[\rho_k(p)=(\frac{k}{2\pi})^n\left(1+\sum_{v\in \Lambda\setminus \{0\}}e^{-\frac{k}{4}(|v|_{H})^2}\cos\big(2\pi k\alpha_v(p)\big)\right). \]

\end{theorem}
\begin{remark}
\begin{itemize}
\item This formula is closely parallel to Theorem~1 of \cite{sun-rem-hyp}, and the proof follows the same strategy. Equivalently, the sum may be written over all geodesic loops based at \(p\), as in \cite{sun-rem-hyp}. It is instructive to compare the two settings: in the hyperbolic case the contribution of a loop of length \(l(\gamma)\) is essentially \(\cosh^{-2k}\!\big(\tfrac{l(\gamma)}{2}\big)\), while in the flat (abelian) case it is the Gaussian factor \(e^{-\frac{k}{2}(|v|_{H})^2}\). This contrast reflects the difference between negative curvature and zero curvature geometries.
\item The formula is also valid when the rank of $\Lambda$ is not maximal, i.e., for polarized complex tori that are not necessarily compact. For simplicity, we only state the compact case here, as we think people care most about the compact case.
\end{itemize}
\end{remark}

For a general polarized manifold \((X,L)\) one may ask:
\begin{question}
Given another Hermitian polarization \((L',h')\) with \(\Theta_{h'}=\Theta_h\) (so \(c_1(L')=c_1(L)\)), under what conditions can one detect that the Bergman kernels \(\rho_{L',k}\) and \(\rho_{L,k}\) are different?
\end{question}
Since the Tian--Catlin--Zelditch expansion depends only on local curvature data, it cannot distinguish such global differences. As an application of Theorem \ref{thm-main} we obtain the following rigidity statement.

\begin{theorem}\label{thm-second}
	Let $(X,L)$ be a polarized abelian variety. Let $L'$ be another polarization of $X$ with respect to a Hermitian metric $h'$ on $L'$ such that $\Theta_{h'}=\Theta_h$. Suppose that $$\rho_{L',k}=\rho_{L,k}$$ for some $k\geq 1$. Then $kL'$ is isomorphic to $kL$.
\end{theorem}
Similar to \cite{sun-rem-hyp}, another application of Theorem \ref{thm-main} concerns the locations of the global maxima and minima of $\rho_k$. Set
\[
l_1=\min_{v\in \Lambda\setminus\{0\}} |v|_{H}, \qquad
l_2=\min_{\substack{v\in \Lambda\\ |v|_{H}>l_1}} |v|_{H},
\]
and write $S_1=\{v\in\Lambda:\;|v|_H=l_1\}$.

\begin{theorem}\label{thm-max-min}
\begin{enumerate}
\item[(a)] Assume $\Im H(u,v)\in 2\mathbb{Z}$ for all $u,v\in\Lambda$. Then for every $k\ge1$ the Bergman density $\rho_k$ attains its maximum precisely at those points $p\in X$ for which
\[
\mathrm{Hol}_L(\gamma_{p,v})=1\quad\text{for all }v\in\Lambda.
\]

\item[(b)] Assume $\Im H(u,v)\in 2\mathbb{Z}$ for all $u,v\in S_1$. There exist constants $C>0$ and $k_0\in\mathbb{N}$ such that for every $k\ge k_0$ every global maximum of $\rho_k$ lies within a neighborhood of radius
\[
C\exp\!\Big(\tfrac{k}{4}(l_1^2-l_2^2)\Big)
\]
of some point $p$ satisfying
\[
\mathrm{Hol}_{kL}(\gamma_{p,v})=1\quad\text{for all }v\in S_1.
\]

\item[(c)] Let $\{v_1,\dots,v_m\}$ be a maximally linearly independent subset of $S_1$ and suppose $\#S_1=2m$ (so $S_1=\{\pm v_1,\dots,\pm v_m\}$). Then the statement of (b) holds. Moreover, there exists $C'>0$ and $k_0'$ such that for all $k\ge k_0'$ every global minimum of $\rho_k$ lies within a neighborhood of radius
\[
C'\exp\!\Big(\tfrac{k}{4}(l_1^2-l_2^2)\Big)
\]
of some point $p$ with
\[
\mathrm{Hol}_{kL}(\gamma_{p,v})=-1\quad\text{for all }v\in S_1.
\]
\end{enumerate}
\end{theorem}

\begin{remark}
\begin{itemize}
\item If $(v_1,\dots,v_{2n})$ is any basis of $\Lambda$, Proposition \ref{prop-surj} implies there exist points $p$ with $\mathrm{Hol}_L(\gamma_{p,v_j})=1$ for all $j$. Hence the hypothesis $\Im H(u,v)\in2\mathbb{Z}$ on $\Lambda\times\Lambda$ ensures existence of points where $\mathrm{Hol}_L(\gamma_{p,v})=1$ for every $v\in\Lambda$.
\item In part (b) it is enough to verify the holonomy condition $\mathrm{Hol}_{kL}(\gamma_{p,v_j})=1$ for a maximally independent subset $\{v_j\}_{j=1}^m\subset S_1$.
\end{itemize}
\end{remark}
Similar to Conjecture 1.6 in \cite{sun-rem-hyp}, we propose the following rigidity conjecture for polarized abelian varieties:

\begin{conjecture}
	Let $(X,L)$ and $(X',L')$ be polarized abelian varieties of dimension $n$. Suppose there is a diffeomorphism $\Phi:X\to X'$ and an integer $k\geq 1$ such that
	\[
	\rho_{L,k}=\Phi^*\rho_{L',k}.
	\]
	Then $\Phi$ is either a biholomorphism or an anti-biholomorphism, and moreover
	\[
	\Phi^*((L')^{\!k})\cong L^{\!k}\qquad\text{or}\qquad \Phi^*((L')^{\!k})\cong\overline{L^{\!k}},
	\]
	respectively. Here $\overline{L}$ denotes the complex-conjugate line bundle of $L$ (transition functions given by complex conjugates).
\end{conjecture}

Theorem \ref{thm-second} gives partial evidence for this conjecture.

Lastly, we state an off-diagonal decay estimate for the Bergman kernel. For $x,y\in X$ let $\mathfrak{G}_{x,y}$ denote the set of geodesic segments joining $x$ to $y$ (parameterized by arc length).
\begin{theorem}\label{thm-off-diag}
	Let $(X,\omega,L,h)$ be a polarized abelian variety as in Theorem \ref{thm-main}. For every $k\ge 1$ and every $x\neq y$,
	\[
	|K_k(x,y)|_{h^k}\;\le\;\Big(\frac{k}{2\pi}\Big)^{\!n}\sum_{\gamma\in\mathfrak{G}_{x,y}} e^{-\frac{k}{4}\,\ell(\gamma)^2},
	\]
	where $\ell(\gamma)$ is the length of $\gamma$.
\end{theorem}

The corresponding result in \cite{sun-rem-hyp} was the first to relate the global off-diagonal decay to the geodesic distance (compared to many local results, see the references therein); the theorem above provides a second instance, now in the flat setting.

\medskip
We briefly outline the proofs. The argument for Theorem \ref{thm-main} follows the same two-step strategy as in \cite{sun-rem-hyp}: first establish the formula for the cylinder (quotients of $\C^n$ by a cyclic subgroup), and then lift the global Bergman kernel to the universal cover and express it as a summation over the lattice $\Lambda$. This strategy parallels the construction behind the Selberg trace formula. 

Theorem \ref{thm-second} is obtained by integrating the kernel along complementary subtori to isolate the contributions of a chosen lattice direction; comparison of these averaged formulas recovers the holonomy data. Theorems \ref{thm-max-min} and \ref{thm-off-diag} are direct consequences of Theorem \ref{thm-main}, while the off-diagonal decay is obtained by estimating the contribution of peak sections and summing over geodesic segments joining two points.

\medskip

\noindent\textbf{Organization of the paper.}
In Section~2 we collect background material and derive the one-dimensional cylinder formula. Section~\ref{section-cylinder} treats the higher-dimensional cylinder and proves the cylinder case of Theorem~\ref{thm-main}. Section~\ref{sec-proof-thm-main} carries out the lattice summation on the universal cover and completes the proof of Theorem~\ref{thm-main}; the remaining theorems are then proved as applications.

\medskip

\noindent\textbf{Acknowledgements.} The author would like to thank Professor Song Sun for many very helpful discussions. 

\section{Background and preparations}
For the reader's convenience, we copy some background materials about Bergman kernels and holonomy of line bundles from \cite{sun-rem-hyp}.
\subsection*{Bergman kernel}
Let $(L,h)$ be a Hermitian holomorphic line bundle over a Kähler manifold $(X,\omega)$. The off-diagonal Bergman kernel of $\hcal_k$ (defined in the introduction) is defined as follows. 

For each point $y\in X$, valuation map $\val_y:\hcal_k\to L^k_y$ is a bounded functional. So there is a unique $L^2$-integrable holomorphic section $K_y$ of $L^k$ with values in $L^k_y$ such that \begin{equation}
	\langle s,K_y \rangle=s(y), \quad \forall s\in \hcal_k.
\end{equation}
So the Bergman kernel is defined as $K(x,y)\eqd K_y(x)$. If $\{s_i\}_{1\leq i<\infty}$ is an orthonormal basis of $\hcal_k$, then \[K(x,y)=\sum_{i=1}^{\infty}s_i(x)\otimes \bar{s}_i(y).\]
From the definition of $K(x,y)$, this summation is independent of the choice of the orthonormal basis.

\subsection*{Holonomy of line bundles}
Let $(L,h)\to M$ be a Hermitian line bundle endowed with a unitary connection $\nabla$. Fix a base point $x\in M$ and let $\gamma:[0,1]\to M$ be a oriented loop with $\gamma(0)=\gamma(1)=x$.

If $\gamma$ is piecewise $C^1$, parallel transport along $\gamma$ is defined by solving the ordinary differential equation
\[
\nabla_{\dot\gamma(t)}s(t)=0,\qquad s(0)=v\in L_x,
\]
which yields a unique section $s(t)$ along $\gamma$. The parallel transport map
\[
P_\gamma:\;L_x\to L_x,\qquad P_\gamma(v)\eqd s(1),
\]
is a unitary linear map of the fibre $L_x$, hence multiplication by a unit complex number. The holonomy of $(L,\nabla)$ along $\gamma$ is the element
\[
\operatorname{Hol}_\nabla(\gamma)\in\mathrm{U}(1)
\]
characterising $P_\gamma$, i.e. $P_\gamma=\operatorname{Hol}_\nabla(\gamma)\cdot\mathrm{Id}_{L_x}$. Since the Chern connection of a holomorphic line bundle is determined by the Hermitian metric $h$, we also use $\operatorname{Hol}_h(\gamma)$ or $\operatorname{Hol}_L(\gamma)$ to denote the holonomy of the Chern connection of $(L,h)$ along $\gamma$.
\subsection*{Poisson summation formula}

\begin{theorem}[Poisson summation formula]
Let $\hat{g}$ be the Fourier transform of $g$. Suppose $g(x) = \int_{\R} \hat{g}(y) e^{2\pi i x y} dy$ with $|g(x)| \leq A(1+|x|)^{-1-\delta}$ and $|\hat{g}(y)| \leq A(1+|y|)^{-1-\delta}$ for some $\delta > 0$. Then
\[
\sum_{c=-\infty}^{\infty} g(x + c) = \sum_{\xi=-\infty}^{\infty} \hat{g}(\xi) e^{2\pi i \xi x}.
\]
\end{theorem}
\subsection{Abelian varieties}

Let $V$ be a complex vector space of dimension $n$. A lattice $\Lambda$ in $V$ is a discrete subgroup of $V$ such that $\Lambda \cong \Z^{2n}$ and $\Lambda$ spans $V$ over $\R$. A complex torus is a quotient $X=V/\Lambda$. A complex torus $X$ is an abelian variety if it admits a polarization, i.e., a positive line bundle $L\to X$. In some literatures, polarization of $X$ is also referred to the first Chern class $c_1(L)$, which is given by a positive definite Hermitian form $H$ on $V$ such that the Riemann form $E=\Im H$ satisfies the integrality condition
\begin{equation}\label{eq-integral}
	\Im H(\Lambda,\Lambda)\subset \Z.
\end{equation}
The first Chern class $c_1(L)$ of $L$ is then represented by the K\"ahler form $ \frac{\sqrt{-1}}{2} \sum H_{ij} dz^i\wedge d\bar{z}^j$. 

We know that $E$ is a non-degenerate alternating form, and $H(u,v)=E(iu,v)+iE(u,v)$ for $u,v\in V$ (see \cite{lange2013complex} for more details). For $v\in V$, we also use $|v|$ for $|v|_H$ for simplicity.

\subsection{Flat model}
Let $z=(z_1,\dots,z_n)$ denote the standard complex coordinates on $\C^n$ and set
\[
\omega_0=\frac{\sqrt{-1}}{2}\sum_{i=1}^n dz_i\wedge d\bar z_i,\qquad
\phi_0=\tfrac12|z|^2,\qquad h_0=e^{-\phi_0}=e^{-\frac{1}{2}|z|^2}.
\]
Then $\iddbar\log h_0=\omega_0$. For each positive integer $k$ consider the weighted Bergman space
\[
\hcal_k=\Big\{f\in\ocal(\C^n):\int_{\C^n}|f|^2\,h_0^k\frac{\omega_0^n}{n!}<\infty\Big\},
\]
with the $L^2$--inner product induced by $h_0^k$ and the volume form $\omega_0^n/n!$. The corresponding Bergman kernel function is constant and equals
\[
\rho_k(p)=\Big(\frac{k}{2\pi}\Big)^n\qquad(\forall p\in\C^n).
\]

\subsection{Flat cylinder}
Fix $\eta>0$ and consider the punctured plane $\C^*$. Define
\[
\omega_\eta=\frac{\eta^2\sqrt{-1}}{2}\,\frac{dz\wedge d\bar z}{|z|^2},\qquad
h_\eta=e^{-\eta^2(\log|z|)^2}.
\]
A direct calculation gives $\iddbar\log h_\eta=\omega_\eta$. In the Riemannian metric induced by $\omega_\eta$ every circle $\{|z|=r\}$ has length $2\pi\eta$ (independent of $r$).

Fix $\alpha\in[0,1)$ and, for each integer $k\ge1$, set
\[
\mathfrak{m}_k=k\alpha-\lfloor k\alpha\rfloor \in[0,1).
\]
Let $\nabla$ be the Chern connection of the metric $h_\eta^k$ on the trivial line bundle over $\C^*$, and let $\nabla_{\mathfrak{m}_k}$ be the Chern connection of the metric $h_\eta^k|z|^{-2\mathfrak{m}_k}$. If $\gamma$ is a counterclockwise-oriented circle $\{|z|=r\}$, then the arguments of the holonomies of $\nabla$ and $\nabla_{\mathfrak{m}_k}$ along $\gamma$ differ by $2\pi\mathfrak{m}_k$.

Define the twisted Bergman space
\[
\hcal_{\eta,\alpha,k}=\Big\{f\in\O(\C^*):\int_{\C^*}|f|^2\,h_\eta^k|z|^{-2\mathfrak{m}_k}\,\omega_\eta<\infty\Big\},
\]
endowed with the inner product
\begin{equation}\label{eq-inner-product}
\langle f,g\rangle=\int_{\C^*} f\bar g\,h_\eta^k|z|^{-2\mathfrak{m}_k}\,\omega_\eta.
\end{equation}
We denote by $\rho_{\eta,\alpha,k}$ the Bergman kernel function of $\hcal_{\eta,\alpha,k}$ and refer to it as the twisted Bergman kernel on the flat cylinder $(\C^*,\omega_\eta)$ with parameter $\alpha$.

Clearly, $\{z^a\}_{a\in \Z}$ is an orthogonal basis for $\hcal_{\eta,\alpha,k}$. Let
\begin{equation}\label{eq-Ia}
	I_a=\int_{\C^*} |z|^{2a}h_\eta^k|z|^{-2\mathfrak{m}_k}\omega_\eta.
\end{equation}
Let $t=\log|z|$.
Then we have \begin{eqnarray*}
	I_a&=&2\pi \eta^2\int_{-\infty}^{\infty} e^{2(a-\mathfrak{m}_k)t}e^{-k\eta^2 t^2} dt\\
	&=&2\pi \eta^2 e^{\frac{1}{k\eta^2}(a-\mathfrak{m}_k)^2}\int_{-\infty}^{\infty}e^{-k\eta^2 \tau^2} d\tau\\
	&=&2\pi \eta^2 e^{\frac{1}{k\eta^2}(a-\mathfrak{m}_k)^2}\sqrt{\frac{\pi}{k\eta^2}}.
\end{eqnarray*}
Thus, we have \[\rho_{\eta,\alpha,k}(z)=h_\eta^k|z|^{-2\mathfrak{m}_k}\sum_{a\in \Z}\frac{|z|^{2a}}{I_a}=h_\eta^k|z|^{-2\mathfrak{m}_k}\frac{1}{2\pi}\sqrt{\frac{k}{\pi}}\frac{1}{\eta}\sum_{a\in \Z}e^{-\frac{1}{k\eta^2}(a-\mathfrak{m}_k)^2+2at}.\]
Let \begin{eqnarray*}
	F_t(\xi)&=&\int_{-\infty}^{\infty} e^{-2\pi i\xi a}e^{-\frac{1}{k\eta^2}(a-\mathfrak{m}_k)^2+2at} da\\
	&=&e^{-2\pi i\xi \mathfrak{m}_k+2t\mathfrak{m}_k}\int_{-\infty}^{\infty}e^{-\frac{1}{k\eta^2}b^2+2(t-\pi i\xi)b} db\\
	&=&e^{-2\pi i\xi \mathfrak{m}_k+2t\mathfrak{m}_k+k\eta^2(t-\pi i\xi)^2}\int_{-\infty}^{\infty}e^{-\frac{1}{k\eta^2}\big(b-k\eta^2(t-\pi i\xi)\big)^2} db.
\end{eqnarray*}
Then by contour integration, we have \begin{eqnarray*}
	F_t(\xi)=\sqrt{\pi k\eta^2}e^{-2\pi i\xi \mathfrak{m}_k+2t\mathfrak{m}_k+k\eta^2(t-\pi i\xi)^2}.
\end{eqnarray*}
Therefore, by Poisson summation formula, we have \begin{eqnarray}
	\rho_{\eta,\alpha,k}(z)&=&\frac{k}{2\pi}h_\eta^k|z|^{-2\mathfrak{m}_k}\sum_{\xi\in \Z}e^{-2\pi i\xi \mathfrak{m}_k+2t\mathfrak{m}_k+k\eta^2(t-\pi i\xi)^2}\\
\label{e-rho-cylinder}	&=&\frac{k}{2\pi}\sum_{\xi\in \Z}e^{-k\eta^2\pi^2\xi^2}e^{-2\pi \xi \mathfrak{m}_k i-2k\eta^2\pi\xi t i}.
\end{eqnarray}

\section{Higher dimensional cylinder} \label{section-cylinder}

Let $0\neq v\in (\C^n,\omega_0)$, where $\omega_0$ is the \kahler form corresponding to the standard Euclidean metric. Let $A_v$ be the isometry on $\C^n$ defined by $u\mapsto u+v$. The quotient space $\C^n/\langle A_v\rangle$ can be identified with $\C^{n-1}\times \C^*$. More precisely, by a unitary action, we can assume that $v=(0,a\sqrt{-1})$ for some $a>0$.
Let $z=(z_1,\dots,z_n)=(z',z_n)$ be the coordinates of $\C^n$.
Let $(w_1,\dots,w_n)$ be coordinates of $\C^{n-1}\times \C^*$. We define a map $Q$ from $\C^n$ to $\C^{n-1}\times \C^*$ by \[ Q(z)=(z_1,\dots,z_{n-1},e^{\frac{2\pi}{a}z_n}).\]
Clearly $Q(A_v z)=Q(z) $ for all $z\in \C^n$. So this induces a map $\tilde{Q}$ from $\C^n/\langle A_v\rangle$ to $\C^{n-1}\times \C^*$. Moreover, if $Q(x)=Q(y) $, it is easy to see that $x-y=mv$, for some integer $m$. So $\tilde{Q}$ is a biholomorphism. We will identify $\C^n/\langle A_v\rangle$ with $\C^{n-1}\times \C^*$ via $\tilde{Q}$. $\omega_0$ then descends to a \kahler form $\omega_v$ on $\C^{n-1}\times \C^*$. Since $\frac{dw_n}{w_n}=\frac{2\pi}{a}dz_n$, we have \begin{eqnarray*}
	\omega_v=\frac{\sqrt{-1}}{2}\sum_{i=1}^{n-1} d w_i\wedge d\bar{w}_i+\frac{\eta^2\sqrt{-1}}{2}\frac{dw_n \wedge d\bar{w}_n}{|w_n|^2},
\end{eqnarray*}
where $\eta=\frac{a}{2\pi}$. Since $\C^{n-1}\times \C^*$ is a Stein manifold, any holomorphic line bundle on it is trivial. We only need to consider the trivial line bundle.

Let \[h_v=e^{-\frac{1}{2}|z'|^2}h_\eta,\] where $h_\eta$ is the Hermitian metric on the trivial bundle on $\C^*$, studies in the subsection ``Flat cylinder''.

Fix a number $\alpha\in [0,1)$, we define $\mathfrak{m}_k=k\alpha-[k\alpha]$ for each integer $k\geq 1$. We study the Bergman space $\hcal_{v,\alpha,k}$ of $L^2$-integrable holomorphic functions on $\C^{n-1}\times \C^*$ with respect to the inner product \begin{equation}\label{eq-inner-product-cylinder}
	\langle f,g\rangle=\int_{\C^{n-1}\times \C^*} f\bar{g}h_v^k|w_n|^{-2\mathfrak{m}_k}\frac{\omega_v^n}{n!}.
\end{equation}
 Let $\rho_{v,\alpha,k}$ be the Bergman kernel function of $\hcal_{v,\alpha,k}$. We call $\rho_{v,\alpha,k}$ the twisted Bergman kernel function on the cylinder $(\C^{n-1}\times \C^*,\omega_v)$ with parameter $\alpha$. Let $\nabla_{v,\alpha,k}$ be the Chern connection of the metric $h_v^k|w_n|^{-2\mathfrak{m}_k}$ on the trivial line bundle over $\C^{n-1}\times \C^*$.
 
 The main result of this section is the following theorem.
 \begin{theorem}\label{thm-cylinder-main}
	Theorem \ref{thm-main}, with $\Lambda=\langle A_v\rangle$ and $\nabla_{v,\alpha,k}$ as the Chern connection, holds for $\hcal_{v,\alpha,k}$.
 \end{theorem}
 
 Clearly, $\{w'^{a} w_n^b\}_{a\in \N^{n-1},b\in \Z}$ is an orthogonal basis for $\hcal_{v,\alpha,k}$, where \[w'^{a}=\prod_{j=1}^{n-1}w_j^{a_j}.\]
 Let \begin{eqnarray*}
	J_{a,b}&=&\int_{\C^{n-1}\times \C^*} |w'^{a}w_n^b|^{2}h_v^k|w_n|^{-2\mathfrak{m}_k}\frac{\omega_v^n}{n!}\\
	 &=& I_b \int_{\C^{n-1}}\prod_{j=1}^{n-1}|w_j|^{2a_j}e^{-\frac{k}{2}|w'|^2}(\frac{\sqrt{-1}}{2})^{n-1}dw_1\wedge d\bar{w}_1\wedge \cdots dw_{n-1}\wedge d\bar{w}_{n-1},
\end{eqnarray*}
where $I_b$ is defined in \eqref{eq-Ia}. Let $|a|=\sum_{j=1}^{n-1}a_j$ and $a!=\prod_{j=1}^{n-1}a_j!$, then we have \[J_{a,b}=I_b\frac{(2\pi)^{n-1}2^{|a|}a!}{k^{|a|+n-1}}.\]
Therefore, we have \begin{eqnarray*}
	\rho_{v,\alpha,k}(w)&=&e^{-\frac{k}{2}|z'|^2}h^k_\eta|w_n|^{-2\mathfrak{m}_k}\sum_{a\in \N^{n-1},b\in \Z}\frac{|w'^{a}|^2|w_n^b|^2k^{|a|+n-1}}{I_b(2\pi)^{n-1}2^{|a|}a!}\\
	&=&e^{-\frac{k}{2}|z'|^2}h^k_\eta|w_n|^{-2\mathfrak{m}_k}\frac{k^{n-1}}{(2\pi)^{n-1}}e^{\frac{k}{2}|z'|^2}\sum_{b\in \Z}\frac{|w_n^b|^2}{I_b}.
\end{eqnarray*} Thus, by \eqref{e-rho-cylinder}, we have \begin{equation}\label{eq-rho-cylinder-final}
	\rho_{v,\alpha,k}(w)=\frac{k^{n}}{(2\pi)^{n}}\sum_{\xi\in \Z}e^{-k\eta^2\pi^2\xi^2}e^{-2\pi \xi \mathfrak{m}_k i-2k\eta^2\pi\xi t i}.
\end{equation}
For an oriented loop $\gamma$ in $\C^{n-1}\times \C^*$, we define the winding number of $\gamma$ as \[\frac{1}{2\pi i}\int_\gamma\frac{dw_n}{w_n}. \]

\begin{lemma}
	Let $p$ be a point with coordinates $w$. Let $\gamma$ be an oriented geodesic loop in $\C^{n-1}\times \C^*$ based at $p$ with winding number $-\xi$. Then the holonomy of $\nabla_{v,\alpha,k}$ along $\gamma$ equals \[e^{-2\pi \xi \mathfrak{m}_k i-2k\eta^2\pi\xi t i}. \]
\end{lemma}

\begin{proof}
	Let $\beta_y$ be the circle $\{w\in \C^{n-1}\times \C^*\big|w'=y,|w_n|=1 \}$, considered as a loop based at $(y,1)$ with counterclockwise orientation. Clearly, the holonomy of $\nabla_{v,\alpha,k}$ along $\beta_y$ equals $e^{2\pi \mathfrak{m}_k i}$.
	When $\xi=-1$, the difference between the arguments of the holonomies of $\nabla_{v,\alpha,k}$ along $\gamma$ and along $\beta_{w'}$ is given by $k$ times the area of the region on $\{w'\}\times \C^*$ bounded by $\gamma$ and $\beta_{w'}$, which equals $k(2\pi \eta)(t\eta)=2k\pi \eta^2 t$. This proves the case when $\xi=-1$. For the general case, one just need to lift the geodesic loops to line segments in $\C^n$. So we have prove the lemma.
\end{proof}

\begin{proof}[Proof of Theorem \ref{thm-cylinder-main}]
	Since $|v|=2\pi\eta $, the theorem follows directly from the above lemma and Formula \ref{eq-rho-cylinder-final}.
\end{proof}

\section{Proof of the main results}\label{sec-proof-thm-main}
The proof of Theorem \ref{thm-main} is similar to the proof of Theorem 1 in \cite{sun-rem-hyp}. For the reader's convenience, we include the details here.

Let $(X,\omega,L,h)$ satisfy the setting of Theorem \ref{thm-main}. Let $\nabla$ be the Chern connection of $(L,h)$.
\subsection{Back to $\C^n$}\label{subsec-back-cn}
Let $\pi:\C^n\to X$ be the quotient map. We can choose coordinates $z=(z_1,\dots,z_n)$ on $\C^n$ such that \[\pi^*\omega=\frac{\sqrt{-1}}{2}\sum_{i=1}^n d z_i\wedge d\bar{z}_i.\]
We denote by $\omega_0=\frac{\sqrt{-1}}{2}\sum_{i=1}^n d z_i\wedge d\bar{z}_i$, $\phi_0=\frac12 |z|^2$.

$\pi^*L$ is defined as the subset \[ \pi^*L=\{(p,v)\in \C^n\times L\big| v\in L_{\pi(p)}\}.\] The Hermitian metric $\pi^*h$ on $\pi^*L$ is the pull-back of $h$.

Since $\pi^*L$ is a holomorphic line bundle on $\C^n$, it is trivial. Let $\bm{e}'$ be a holomorphic frame of $\pi^*L$ on $\C^n$. Then $\parallel \bm{e}'\parallel^2_{\pi^*h}=e^{-\varphi}$ for some smooth function $\varphi$ on $\C^n$ satisfying $\iddbar \varphi=\pi^*\omega=\omega_0$. Then $\varphi-\phi_0$ is harmonic. So there exists a holomorphic function $F$ on $\C^n$ such that $\varphi-\phi_0=\Re F$. Let $\bm{e}=e^{F/2}\bm{e}'$, then \[\parallel \bm{e}\parallel^2_{\pi^*h}=e^{-\phi_0}.\]
Similarly, for any $k\in \mathbb{N}$, we can define a frame $\bm{e}_k$ of $\pi^*L^k$ on $\C^n$ such that $\parallel \bm{e}_k\parallel^2_{\pi^*h}=e^{-k\phi_0}$.

Each $A_v$, for some $v\in \Lambda$, is lifted to a map \[\tilde{A_v}:\pi^*L^k\to \pi^*L^k, \quad (x,v)\mapsto (A_v(x),v).\] 
So for any $s\in H^0(\C^n,\pi^*L^k)$, we can define the pullback map $A_v^*:H^0(\C^n,\pi^*L^k)\to H^0(\C^n,\pi^*L^k)$ by 
\begin{equation}
	A_v^*s(x)=\tilde{A_v}^{-1} s(A_v(x)).
\end{equation}
If $f\in H^0(\C^n,\pi^*L^k)$, then $f=U\bm{e}_k$ for some holomorphic function $U$. Then we have
\[A_v^*s(x)=(A_v^*U\frac{A_v^*\bm{e}_k}{\bm{e}_k})\bm{e}_k,\]where $A_v^*U=U\circ A_v$.

Let \[\hcal_{X,k}=\{s\in H^0(X,L^k): \int_X |s|^2_h \frac{\omega^n}{n!}<\infty\},\] and \[\hcal_{\C^n,k}=\{s\in H^0(\C^n,\pi^*L^k): \int_{\C^n} |s|^2_{\pi^*h} \frac{\omega_0^n}{n!}<\infty\}.\] Let $K_{X,k}$ and $K_k$ be the Bergman kernels of $\hcal_{X,k}$ and $\hcal_{\C^n,k}$ respectively. Let $\rho_{X,k}$ and $\rho_{k}$ be the corresponding Bergman kernel functions.
Let $W_k=\pi^*\hcal_{X,k}\subset H^0(\C^n,\pi^*L^k)$. Clearly, for $s\in W_k$, $A_v^*s=s, \forall v\in \Lambda$. Let $F_\Lambda$ be a fundamental domain of the action of $\Lambda$ on $\C^n$. Then for any $s\in H^0(\C^n,\pi^*L^k)$ that is invariant under the action of $\Lambda$, $s\in V_k$ if and only if \[\int_{F_\Lambda} |s|^2_{\pi^*h}\frac{w_0^n}{n!}<\infty.\] 
With the $L^2$-norm defined by integrating over $F_\Lambda$, $W_k$ is then a Hilbert space.
$\pi^*K_{X,k}$ is then the Bergman kernel of $W_k$ in the sense that, for any $s\in W_k$, \[\int_{F_\Lambda} (s(z),\pi^*K_{X,k}(z,w))_{\pi^*h}\frac{w_0^n}{n!}(z)=s(w),\]
for any $w\in \C^n$.
\subsection{Proof of Theorem \ref{thm-main}}

We can give $\Lambda$ an order according to the norm $|v|$. We denote by $s_w$ the peak section of $\hcal_{\C^n,k}$ at $w\in \C^n$.
It is easy to see that $s_0=(\frac{k}{2\pi})^{n/2}\bm{e}_k$ is a peak section at $0\in \C^n$ for $\hcal_{\C^n,k}$. So $A_v^*s_0$ is a peak section at $-v\in \C^n$.
\begin{lemma}\label{lem-sum-g-f-0}
	Let $f_0=\bm{e}_k$. Then when $k>0$, $\sum_{v\in \Lambda}A_v^*f_0$ converges locally uniformly and absolutely.
\end{lemma}
\begin{proof}
	We have $|f_0(z)|_h=e^{-\frac{k}{4}|z|^2}=e^{-\frac{k}{4}d^2(0,z)}$. So for each $v\in \Lambda$, \[|A_v^*f_0(z)|_h=e^{-\frac{k}{4}d^2(-v,z)}.\]
	Since the number $N_M\eqd\#\{v\in \Lambda\big| |v|<M \}$ grows at a polynomial speed, it is easy to get the conclusion of the lemma.
\end{proof}

\begin{lemma}\label{lem-sum-g-f-0-in-v-k}
	For $k>0$, we have $\sum_{v\in \Lambda}A_v^*f_0\in W_k$. 
\end{lemma}
\begin{proof}
	It suffices to show that \[\sum_{v\in \Lambda}(\int_{F_\Lambda} |A_v^*f_0|_h^2\frac{w_0^n}{n!})^{1/2}<\infty. \]	
	We have \[\int_{d(0,z)>R}|A_v^*f_0|_h^2\frac{w_0^n}{n!}<C(kR^2)^{n-1}e^{-kR^2},  \] for some constant $C$ depending only on $n$. So the lemma follows from the polynomial growth of $N_M$ as in the proof of Lemma \ref{lem-sum-g-f-0}.
\end{proof}

Let $K_k(z,w)$ be the Bergman kernel of $\hcal_{\C^n,k}$. Then we can write \[K_k(z,w)=\sum_{i=1}^{\infty}f_i(z)\otimes \bar{f}_i(w),\] where $\{f_i\}_{i=1}^\infty$ is an orthonormal basis of $\hcal_{\C^n,k}$.
For each $v\in \Lambda$, we define \[K^v_k(z,w)=\sum_{i=1}^{\infty}A_v^*f_i(z)\otimes \bar{f}_i(w).\]
\begin{proposition}
	$K^v_k(z,w)=\sum_{i=1}^{\infty}A_v^*f_i(z)\otimes \bar{f}_i(w)$ is independent of the choice of orthonormal basis $\{f_i\}_{i=1}^\infty$ of $\hcal_{\C^n,k}$.
\end{proposition}
\begin{proof}
	For any $f\in \hcal_{\C^n,k}$, we have $f=\sum_{i=1}^{\infty}c_i A_v^*f_i$, since $\{A_v^*f_i\}$ is also an orthonormal basis of $\hcal_{\C^n,k}$. Then for any fixed $w$, \[\int_{\C^n} (f,\sum_{i=1}^{\infty}A_v^*f_i(z)\otimes \bar{f}_i(w))_{\pi^*h}\frac{w_0^n}{n!}=\sum_{i=1}^{\infty}c_i f_i(w)=(A_v^{-1})^*f(w), \]
	namely, $K^v_k(z,w)$ is the representation element for the functional $\val_w\circ A_{-v}^*$, hence independent of the choice of the orthonormal basis.
\end{proof}
\begin{defandthm}\label{def-and-thm-k-gamma}
	We define
\[K_{k}^\Lambda(z,w)\eqd \sum_{v\in \Lambda}K^v_k(z,w). \]
Then for $k>0$, and for each fixed $w$, it converges locally uniformly and absolutely. Moreover, we have that for $k>0$, $\sum_{v\in \Lambda}K^v_k(z,w)\in W_k$ for each fixed $w$.
\end{defandthm}
\begin{proof}
	Since $K^v_k(z,w)$ is independent of the choice of orthonormal basis, for each fixed $w$, we can choose the orthonormal basis $\{f_i\}$ such that $f_1=s_w$, a peak section at $w$. So $f_i(w)=0$ for $i>1$. So we have \[K^v_k(z,w)=A_v^*s_w(z)\otimes \bar{s}_w(w).\]
	So by Lemma \ref{lem-sum-g-f-0}, we have $\sum_{v\in \Lambda}K^v_k(z,w)$ converges locally uniformly and absolutely. And by Lemma \ref{lem-sum-g-f-0-in-v-k}, $\sum_{v\in \Lambda}K^v_k(z,w)\in W_k$ for each fixed $w$.
\end{proof}
\begin{lemma}\label{lem-int-s-sum-g-s-p}
	For any $s\in W_k$, we have 
	\[\int_{F_\Lambda}(s,\sum_{v\in \Lambda}A_v^*s_p)_{\pi^*h}\frac{\omega_0^n}{n!}=(\frac{2\pi}{k})^{n/2}\frac{s(p)}{s_p(p)}.\]
\end{lemma}
\begin{proof}
	It suffices to show this for $p=0$. By Lemmas \ref{lem-sum-g-f-0} and \ref{lem-sum-g-f-0-in-v-k}, 
	\[ \int_{F_\Lambda}(s,\sum_{v\in \Lambda}A_v^*s_0)_{\pi^*h}\frac{\omega_0^n}{n!}=\sum_{v\in \Lambda}\int_{F_\Lambda}(s,A_v^*s_0)_{\pi^*h}\frac{\omega_0^n}{n!}.\]
	On the other hand, since $s_0=(\frac{k}{2\pi})^{n/2}\bm{e}_k$ is $L_1$ integrable on $\C^n$, and $s$ is bounded on $\C^n$, we have $(s,s_0)_{\pi^*h}\in L^1(\C^n,\frac{\omega_0^n}{n!})$. So \[\sum_{v\in \Lambda}\int_{F_\Lambda}(s,A_v^*s_0)_{\pi^*h}\frac{\omega_0^n}{n!}=\int_{\C^n} (s,s_0)_{\pi^*h}\frac{\omega_0^n}{n!}.\]
 Then by taking the Taylor expansion of $\frac{s}{\bm{e}_k}$, we get the right hand side equals $(\frac{2\pi}{k})^{n/2}\frac{s(0)}{\bm{e}_k(0)}=\frac{s(0)}{s_0(0)}$.
\end{proof}
\begin{theorem}\label{thm-k-gamma-bergman}
	For $k>0$, $K_{k}^\Lambda(z,w)$ is the Bergman kernel of $W_k$.
\end{theorem}
\begin{proof}
 For fixed $w$, we have \[K_{k}^\Lambda(z,w)=(\sum_{v\in \Lambda}A_v^*s_w(z))\otimes \bar{s}_w(w).\]
	So for each $S\in \hcal_{X,k}$, we have \[\int_{F_\Lambda} (S,K_{k}^\Lambda(z,w))_{\pi^*h}\frac{\omega_0^n}{n!}=\frac{S(w)}{s_w(w)}s_w(w)=S(w).\]
	We have proved the theorem.
\end{proof}
\begin{proof}[Proof of Theorem \ref{thm-main}]
	 For a point \(p\in X\) choose a lift \(\tilde p\in\C^n\); without loss of generality we may take \(\tilde p=0\). By Theorem \ref{thm-k-gamma-bergman} we have the identity
    \[
    \rho_{X,k}(0)=K_{k}^\Lambda(0,0)
    = (\frac{k}{2\pi})^n \;+\; \sum_{v\in\Lambda\setminus\{0\}} K^v_k(0,0).
    \]
	Recall that an element $v\in \Lambda$ is primitive if it can not be written as $v=mu$ for some $u\in\Lambda$ and $m>1$.
    Let \(\mathcal P\) denote the set of primitive elements of \(\Lambda\). Writing \(\langle v\rangle=\{mv:\,m\in\Z \}\) for the cyclic subgroup generated by \(v\in\mathcal P\), we may reorganise the sum as
    \[
    \sum_{v\in\Lambda\setminus\{0\}} K^v_k(0,0)
    = \frac12\sum_{v\in\mathcal P}\sum_{u\in\langle v\rangle\setminus\{0\}} K^{u}_k(0,0).
    \]

    Fix \(v\in\mathcal P\). The quotient \(\C^n/\langle v\rangle\) is a cylinder studied in Section \ref{section-cylinder}.Let
    \(\pi_v:\C^n\to \C^n/\langle v\rangle\) be the quotient map and \(\chi_v:\C^n/\langle v\rangle\to X\) the natural covering with \(\pi=\chi_v\circ\pi_v\). For each $u\in \langle v\rangle$, let $e^{2\pi \alpha_{u}(p)\sqrt{-1}}$ be the holonomy of $L$ along the geodesic loop $\gamma_{p,u}$. Since the pull-back of a parallel section is still a parallel section, the holonomy of the connection $\chi_v^*\nabla$ on $\chi_v^*L$ along the geodesic loop based at $\pi_v(0)$ corresponding to $u$ and is also $e^{2\pi \alpha_{u}(p)\sqrt{-1}}$.

    Applying Theorem \ref{thm-cylinder-main} to the quotient \(\C^n/\langle v\rangle\) yields, for each \(v\in\mathcal P\),
    \[
    \sum_{u\in\langle v\rangle\setminus\{1\}} K^{u}_k(0,0)
    = (\frac{k}{2\pi})^n\sum_{u\in\langle u\rangle\setminus\{0\}}
    e^{-k(\frac{|u|}{2})^2}\cos\!\big(2\pi k\alpha_{u}(p)\big).
    \]
    Summing over all primitive elements \(g\in\mathcal P\) gives the formula in Theorem \ref{thm-main}. 
\end{proof}
\begin{proof}[Proof of Theorem \ref{thm-off-diag}]
	By Theorem \ref{thm-k-gamma-bergman}, we have 
	\[\pi^*K_{X,k}(z,w)=\sum_{v\in \Lambda}(A_v^*s_w(z))\otimes \bar{s}_w(w).\]
	We can assume that $w=0$. Then since $|A_v^*s_0(z)|=(\frac{k}{2\pi})^{n/2}e^{-\frac{k}{4}d^2(-v,z)}$. We have \[|\pi^*K_{X,k}(z,0)|_h\leq (\frac{k}{2\pi})^{n}\sum_{v\in \Lambda}e^{-\frac{k}{4}d^2(-v,z)}.\]
	Since $d(-v,z)$ can be identified with the length of the corresponding geodesic segment in $X$ joining $\pi(z)$ and $\pi(0)$, the theorem follows.
\end{proof}
\subsection{Proof of Theorem \ref{thm-second}}
For a fixed line bundle $L$ on $X$ with $\omega\in 2\pi c_1(L)$, let $h$ be a Hermitian metric on $L$ such that $\Theta_h=-i\omega$. Let $F$ denote the underlying $C^\infty$ complex line bundle of $L$. Let $\ucal$ be the space of unitary connections on $(F,h)$ whose curvatures satisfy the condition that the $(0,2)$-part is $0$. Let $\nabla_0\in \ucal$ be the Chern connection of $(L,h)$.

Let $\jcal_F$ be the space of holomorphic structures on $F$. The space of isomorphism classes of holomorphic structures on $F$ can be identified with $\Pic^0(X)+L$. Given $\nabla\in \ucal$, the $(0,1)$-part $\nabla^{0,1}$ of $\nabla$ defines a holomorphic structure on $F$, by requiring that a local smooth section $s$ of $F$ is holomorphic if and only if $\nabla^{0,1}s=0$.
This defines a surjective map $\kappa$ from $\ucal$ to $\jcal_F$.

\begin{lemma}\label{lem-holonomy}
	Let $\{v_1,\dots,v_{2n}\}$ be a basis of $\Lambda$. Let $L'\in \jcal_F$, $p\in X$. Suppose that $\text{Hol}_{L'}(\gamma_{p,v_i})=\text{Hol}_{L'}(\gamma_{p,v_i})$, $1\leq i\leq 2n$. Then $L'$ is isomorphic to $L$.
\end{lemma}
\begin{proof}
	$L'=\kappa(\nabla')$ for some $\nabla'\in \ucal$. Then $\alpha=\frac{1}{2\pi i}(\nabla'-\nabla_0)$ is a closed real $1$-form on $X$. Since the holonomies of $\nabla'$ and $\nabla_0$ along $\gamma_{p,v_i}$ are equal for all $i$, we have \[\int_{v_i}\alpha\in \Z\] for all $i$. Therefore, the cohomology class $[\alpha]$ is in the image of the map $H^1(X,\Z)\to H^1(X,\R)$. So there exists a smooth complex gauge transformation $g:X\to S^1$ such that $\nabla'=\nabla_0+g^{-1}dg$. Thus, $L'$ is isomorphic to $L$.
\end{proof}




It is easy to see that if $u,v$ are linearly dependent vectors, then $\text{Hol}_L(\gamma_{q,u})=\text{Hol}_L(\gamma_{p,u})$, for $q\in \gamma_{p,v}$. When $u,v\in \Lambda$ are linearly independent, let $\tilde{p}$ be a lift of $p$ in $\C^n$. For each $q\in \gamma_{p,v}$, there is a lift $\tilde{q}$ contained in the line segment \[\tilde{p}+tv\big| t\in[0,1].\]
	Let $\Sigma_{p,q,u}$ be the parallelogram spanned by $u$ and the line segment $[\tilde{p},\tilde{q}]$, oriented by the order $(v,u)$.
	Then, we have \begin{equation}
	\arg\text{Hol}_L(\gamma_{q,u})-\arg\text{Hol}_L(\gamma_{p,u})=\int_{\Sigma_{p,q,u}}\omega, \quad \mod 2\pi
	\end{equation}
	 for all $q\in \gamma_{p,v}$. Since $\int_{\Sigma_{p,q,u}}\omega=-2\pi \frac{d(p,q)}{|v|}\Im H(v,u)$, we have
	 \begin{equation}\label{eq-hol}
	 \arg\text{Hol}_L(\gamma_{q,u})-\arg\text{Hol}_L(\gamma_{p,u})=-2\pi \frac{d(p,q)}{|v|}\Im H(v,u).
	 \end{equation}
	 So we have:
	 \begin{lemma}\label{lem-cos-vu}
		Let $u,v\in \Lambda$ be linearly independent vectors. Then \[\cos(2\pi \alpha_u(q(s)))=\cos(2\pi (\alpha_u q(0)-\frac{q(s)}{|v|}\Im H(v,u))), \]where $s$ is the length parameter. In particular, if $\Im H(v,u)=0$, then $\text{Hol}_L(\gamma_{q,u})$ is constant for $q\in \gamma_{p,v}$.
	 \end{lemma}

\begin{proof}[Proof of Theorem \ref{thm-second}]
Let $\{v_1,\dots,v_{2n}\}$ be a basis of $\Lambda$. By Lemma \ref{lem-holonomy}, we only need to show that $\text{Hol}_{L'}(\gamma_{p,v_i})=\text{Hol}_{L'}(\gamma_{p,v_i})$, $1\leq i\leq 2n$ for some $p\in X$. For a fixed $p\in X$, we only need to show that $\text{Hol}_{L'}(\gamma_{p,v_1})=\text{Hol}_{L'}(\gamma_{p,v_1})$.

Let $U$ be the real hyperplane of $\C^n$ defined by \[\{u\in \C^n\big|\Im H(v_1,u)=0 \}. \]
Then the image of $U$ in $X$ is a closed sub-torus. Then the quotient $X/\pi(U)$ is also a torus. So $X/\pi(U)\backsimeq S^1$. More precisely, let $Q\backsimeq\R$ be the quotient space $\C^n/U$ and let $\tilde{P}:\C^n\to Q$ be the quotient map. Then the image $\tilde{P}(\Lambda)$ is generated by some $u$ that is a linear combination of $\{v_i\big| \big| \Im H(v_1,v_i)\neq 0, 1\leq i\leq 2n\} $ with integer coefficients. Then $X/\pi(U)$ can be identified with $Q/\langle u\rangle$. Let $P:X\to Q/\langle u\rangle$ be the quotient map. Each fiber $P^{-1}(x)$ is a copy of the torus $\pi(U)$. More precisely, let $\tilde{x}$ be a lift of $x$ to $\C^n$, then $P^{-1}(x)$ is the image $\pi(\tilde{x}+U)$. So $P^{-1}(x)$ inherits a Riemannian metric from $\C^n$. Let $d\mu_x$ be the corresponding volume form on $P^{-1}(x)$. Then we define the push-forward map \[P_*:C(X)\to C(Q/\langle u\rangle) \] by \[P_*(f)(x)=\int_{P^{-1}(x)} fd\mu_x, \]
for each continuous function $f$ on $X$. By Lemma \ref{lem-cos-vu}, $\text{Hol}_{L}(\gamma_{q,v_1})$ is constant for  
$q\in P^{-1}(x)$, for each $x$. Thus, considering $cos(2\pi \alpha_{v_1}(q))$ as a function of $q\in X$, we have \[(P_*\cos(2\pi \alpha_{v_1}(q)))(x)=\vol(P^{-1}(x))\cos(2\pi \alpha_{v_1}(q_0)),  \]where $q_0$ is any point in $P^{-1}(x)$. Since $\vol(P^{-1}(x))$ is independent of $x$, we denote it by $\nu$. Let $x_0=P(p)$ and $\phi_p=\alpha_{v_1}(p)$. We can use $t\in [0,1]$ to parametrize $Q/\langle u\rangle$ such that $x_0$ corresponds to $0$ and such that $t$ is proportional to the length parameter. Then we have \[(P_*\cos(2\pi \alpha_{v_1}(q)))(x)=\nu\cos(2\pi (\lambda t+\phi_p)),\] for some integer $\lambda$. Similarly, for any integer $m\neq 0$, we have 
\[(P_*\cos(2\pi \alpha_{mv_1}(q)))(x)=\nu\cos(2\pi m (\lambda t+\phi_p)).\]

If $u\in \Lambda $ is not a multiple of $v_1$, then $\exists u'\in \Lambda\cup U$ such that $\Im H(u',u)\neq 0$. 
By Lemma \ref{lem-cos-vu}, since $\int_{0}^{1}\cos(2\pi mx)dx=0$ for any integer $m$, we then have \[(P_*\cos(2\pi \alpha_{u}(q)))(x)=0, x\in Q/\langle u\rangle. \] 
Therefore, by uniform convergence, we have \[(P_*(\frac{(2\pi)^n}{k^n}\rho_{L,k}-1))(x)=2\nu \sum_{m\geq 1} e^{-\frac{k}{4}|m|^2|v_1|_H^2}\cos(2\pi m (\lambda t+\phi_p)). \] 

For $L'$, we let $\text{Hol}_{L'}(\gamma_{p,v})=e^{2\pi i \alpha'_v(p)}$ and $\phi'_p=\alpha'_{v_1}(p)$. Then we also have \[(P_*(\frac{(2\pi)^n}{k^n}\rho_{L',k}-1))(x)=2\nu \sum_{m\geq 1} e^{-\frac{k}{4}|m|^2|v_1|_H^2}\cos(2\pi m (\lambda t+\phi'_p)). \]
Therefore, if $\rho_{L',k}=\rho_{L,k}$, we get \[\sum_{m\geq 1} e^{-\frac{k}{2}|m|^2|v_1|_H^2}\cos(2\pi m (\lambda t+\phi'_p))= \sum_{m\geq 1} e^{-\frac{k}{2}|m|^2\|v_1\|^2}\cos(2\pi m (\lambda t+\phi_p)), \] for $t\in [0,1]$. So we get $\phi_p'=\phi_p, \mod 1$, namely $\text{Hol}_{L'}(\gamma_{p,v_1})=\text{Hol}_{L'}(\gamma_{p,v_1})$. So we have proved the theorem.

\end{proof}
\subsection{Proof of Theorem \ref{thm-max-min}}
Let $\vcal=(v_1,\dots,v_{2n})$ be a basis of $\Lambda$. The space of functions $\Hom(\vcal,S^1)$ is naturally a real torus of dimension $2n$.  Let $L$ be a line bundle on $X$ as in the setting of Theorem \ref{thm-main}. We define the following map: \[\Phi_\vcal:X\to \Hom(\vcal,S^1),\quad p\mapsto (\operatorname{Hol}_\nabla(\gamma_{p,v_1}),\cdots,\operatorname{Hol}_\nabla(\gamma_{p,v_{2n}}))\]
\begin{proposition}\label{prop-surj}
	 $\Phi_\vcal$ is surjective.
\end{proposition}
\begin{proof}
	
Write \(E:=\Im H\). 
Fix a basepoint \(0\in X\).
If \(p=[z]\in X\) with \(z\in\C^n\) a lift, then by \eqref{eq-hol}
\[
\operatorname{Hol}_\nabla(\gamma_{p,v_i})
=\operatorname{Hol}_\nabla(\gamma_{0,v_i})e^{2\pi i\,E(v_i,z)}.
\]
Let $\Pi:\R^{2n}\to \Hom(\vcal,S^1)$ be a covering map chosen so that \[\Pi(0)=(\operatorname{Hol}_\nabla(\gamma_{0,v_1}),\cdots,\operatorname{Hol}_\nabla(\gamma_{0,v_{2n}})).\]
Then the map $\Phi_\vcal\circ \pi:\C^n\to \Hom(\vcal,S^1)$ is lifted to a map $\tilde{\Phi}_\vcal:\C^n\to \R^{2n}$ given by
\[z\mapsto 2\pi( E(v_1,z),\cdots,E(v_{2n},z)),\]
which is clearly linear. To show that $\Phi_\vcal$ is surjective, it suffices to show that $\tilde{\Phi}_\vcal$ is surjective.
Since $E$ is non-degenerate, $\ker \tilde{\Phi}_\vcal=0$. So $\tilde{\Phi}_\vcal$ is surjective.

\end{proof}
\begin{proof}[Proof of Theorem \ref{thm-max-min}]
	Part (a) follows directly from Theorem \ref{thm-main}. For part (b) and (c), we have that $\exists C_1>0$, $k_0>0$  such that for $k\geq k_0$, \[\sum_{v\in \Lambda, |v|_H>l_1}e^{-\frac{k}{4}|v|_H^2}\leq C_1 e^{-\frac{k}{4}l_2^2}.\] 
	Let $p$ be a point satisfying the condition that $\text{Hol}_{L}(\gamma_{p,v})=1$ for all $v\in S_1$. Suppose that $k\geq k_0$ and $\rho_k(q)>\rho_k(p)$, then we have \[\sum_{v\in S_1}e^{-\frac{k}{4}l_1^2}(1-\cos(2\pi k\alpha_v(q)))\leq 2C_1 e^{-\frac{k}{4}l_2^2}, \]namely \[\sum_{v\in S_1}(1-\cos(2\pi k\alpha_v(q)))\leq 2C_1 e^{-\frac{k}{4}(l_2^2-l_1^2)}.\] 
	If we require that $\alpha_v(q)\in (-\pi,\pi]$ for all $v$, then this implies that $|\alpha_v(q)|$ is small for $v\in S_1$. 
Let $\{v_1, v_2,\cdots, v_m\}$ be a maximally linearly independent subset of $S_1$.
	From the proof of Proposition \ref{prop-surj}, it is easy to see that $\exists C_2>0$, independent of $k$, such that when $\epsilon$ is small enough, for any $q'\in X$, if $1-\cos(2\pi k\alpha_{v_j}(q'))<\epsilon$, $1\leq j\leq m$, then $\exists p'\in X$ such that $d(p',q')<C_2\epsilon$ and $e^{2\pi k\alpha_{v_j}(p')\sqrt{-1}}=1$, $1\leq j\leq m$. Let $q'=q$, we get $p'$ such that $d(p',q)<2C_2C_1 e^{-\frac{k}{4}(l_2^2-l_1^2)}$ and $e^{2\pi k\alpha_{v_j}(p')\sqrt{-1}}=1$, $1\leq j\leq m$.

	Under the assumption of part (b), this implies that $e^{2\pi k\alpha_{v}(p')\sqrt{-1}}=1$ for all $v\in S_1$. So we have proved part (b).

	The assumption that $\#S_1=2m$ in part (c) just says that $S_1=\{\pm v_1,\cdots,\pm v_m\}$. So the part for maximum follows directly. The part for minimum is similar.
\end{proof}

\bibliographystyle{plain}


\bibliography{references}

\begin{thebibliography}{10}

\bibitem{berman2011fekete}
Robert Berman, S{\'e}bastien Boucksom, and David~Witt Nystr{\"o}m.
\newblock Fekete points and convergence towards equilibrium measures on complex manifolds.
\newblock {\em Acta Mathematica}, 207(1):1--27, 2011.

\bibitem{bsz0}
Pavel Bleher, Bernard Shiffman, and Steve Zelditch.
\newblock Poincar\'{e}-{L}elong approach to universality and scaling of correlations between zeros.
\newblock {\em Comm. Math. Phys.}, 208(3):771--785, 2000.

\bibitem{bsz1}
Pavel Bleher, Bernard Shiffman, and Steve Zelditch.
\newblock Universality and scaling of correlations between zeros on complex manifolds.
\newblock {\em Invent. Math.}, 142(2):351--395, 2000.

\bibitem{boutet1975singularite}
Louis Boutet~de Monvel and Johannes Sj{\"o}strand.
\newblock Sur la singularit{\'e} des noyaux de bergman et de szeg{\"o}.
\newblock {\em Journ{\'e}es {\'e}quations aux d{\'e}riv{\'e}es partielles}, pages 123--164, 1975.

\bibitem{Catlin}
David Catlin.
\newblock The {B}ergman kernel and a theorem of {T}ian.
\newblock In {\em Analysis and geometry in several complex variables ({K}atata, 1997)}, Trends Math., pages 1--23. Birkh\"{a}user Boston, Boston, MA, 1999.

\bibitem{dailiuma}
Xianzhe Dai, Kefeng Liu, and Xiaonan Ma.
\newblock On the asymptotic expansion of bergman kernel.
\newblock {\em Journal of Differential Geometry}, 72(1):1--41, 2006.

\bibitem{donaldson2001}
Simon Donaldson.
\newblock Scalar curvature and projective embeddings, {I}.
\newblock {\em Journal of Differential Geometry}, 59(3):479--522, 2001.

\bibitem{Donaldson15}
Simon Donaldson.
\newblock Algebraic families of constant scalar curvature {K}\"{a}hler metrics.
\newblock In {\em Surveys in differential geometry 2014. {R}egularity and evolution of nonlinear equations}, volume~19 of {\em Surv. Differ. Geom.}, pages 111--137. Int. Press, Somerville, MA, 2015.

\bibitem{Donaldson2014Gromov}
Simon Donaldson and Song Sun.
\newblock Gromov-{H}ausdorff limits of {K}\"{a}hler manifolds and algebraic geometry.
\newblock {\em Acta Mathematica}, 213(1):63--106, 2014.

\bibitem{dsz1}
Michael~R. Douglas, Bernard Shiffman, and Steve Zelditch.
\newblock Critical points and supersymmetric vacua. {I}.
\newblock {\em Comm. Math. Phys.}, 252(1-3):325--358, 2004.

\bibitem{dsz2}
Michael~R. Douglas, Bernard Shiffman, and Steve Zelditch.
\newblock Critical points and supersymmetric vacua. {II}. {A}symptotics and extremal metrics.
\newblock {\em J. Differential Geom.}, 72(3):381--427, 2006.

\bibitem{dsz3}
Michael~R. Douglas, Bernard Shiffman, and Steve Zelditch.
\newblock Critical points and supersymmetric vacua. {III}. {S}tring/{M} models.
\newblock {\em Comm. Math. Phys.}, 265(3):617--671, 2006.

\bibitem{Fefferman1974}
Charles Fefferman.
\newblock The bergman kernel and biholomorphic mappings of pseudoconvex domains.
\newblock {\em Inventiones mathematicae}, 26(1):1--65, Mar 1974.

\bibitem{kerzman1978cauchy}
Norberto Kerzman and Elias~M Stein.
\newblock The cauchy kernel, the szeg{\"o} kernel, and the riemann mapping function.
\newblock {\em Mathematische Annalen}, 236(1):85--93, 1978.

\bibitem{lange2013complex}
Herbert Lange and Christina Birkenhake.
\newblock {\em Complex abelian varieties}, volume 302.
\newblock Springer Science \& Business Media, 2013.

\bibitem{Lu2000On}
Zhiqin Lu.
\newblock On the lower order terms of the asymptotic expansion of {T}ian-{Y}au-{Z}elditch.
\newblock {\em Amer.J.math}, 122(2):235--273, 2000.

\bibitem{MM}
Xiaonan Ma and George Marinescu.
\newblock {\em Holomorphic {M}orse inequalities and {B}ergman kernels}, volume 254 of {\em Progress in Mathematics}.
\newblock Birkh\"{a}user Verlag, Basel, 2007.

\bibitem{sz1}
Bernard Shiffman and Steve Zelditch.
\newblock Distribution of zeros of random and quantum chaotic sections of positive line bundles.
\newblock {\em Comm. Math. Phys.}, 200(3):661--683, 1999.

\bibitem{Shiffman2008NV}
Bernard Shiffman and Steve Zelditch.
\newblock Number variance of random zeros on complex manifolds.
\newblock {\em Geometric and Functional Analysis}, 18(4):1422--1475, Dec 2008.

\bibitem{Sun2011Expected}
Jingzhou Sun.
\newblock Expected {E}uler characteristic of excursion sets of random holomorphic sections on complex manifolds.
\newblock {\em Indiana University Mathematics Journal}, 61(3):pages. 1157--1174, 2011.

\bibitem{sun-rem-hyp}
Jingzhou Sun.
\newblock On the {B}ergman kernel of hyperbolic {R}iemann surfaces.
\newblock {\em Preprint}, arxiv 2511.16240, Available at http://arxiv.org/abs/2511.16240.

\bibitem{Tian1990On}
Gang Tian.
\newblock On a set of polarized {K}\"ahler metrics on algebraic manifolds.
\newblock {\em Journal of Differential Geometry}, 32(1990):99--130, 1990.

\bibitem{Zelditch2000Szego}
Steve Zelditch.
\newblock Szego kernels and a theorem of {T}ian.
\newblock {\em International Mathematics Research Notices}, 6(6):317--331, 2000.

\end{thebibliography}

\end{document}